\documentclass[12pt]{article}

\usepackage[margin=1in]{geometry}
\usepackage[usenames]{xcolor}
\usepackage{amssymb}
\usepackage{relsize}
\usepackage{amsmath}
\usepackage{amsthm}
\usepackage{amsfonts}
\usepackage{amscd}
\usepackage{graphicx}
\usepackage{hyperref}
\usepackage{thmtools}
\usepackage{thm-restate}
\usepackage{tikz-cd} 
\usepackage{float}
\usepackage{scalerel}

\begin{document}

\theoremstyle{plain}
\newtheorem{theorem}{Theorem}
\newtheorem{corollary}[theorem]{Corollary}
\newtheorem{lemma}[theorem]{Lemma}
\newtheorem{proposition}[theorem]{Proposition}

\theoremstyle{definition}
\newtheorem{defn}{Definition}
\newtheorem{example}[theorem]{Example}
\newtheorem{conjecture}[theorem]{Conjecture}
\newtheorem{question}[theorem]{Question}

\theoremstyle{remark}
\newtheorem{remark}[theorem]{Remark}

\begin{center}
\vskip 1cm{\LARGE\bf 
Generalization of a Result of Sylvester Regarding the Frobenius Coin Problem and an Elementary Proof of Eisenstein's Lemma for Jacobi Symbols  \\ 
\vskip .1in

}
\vskip 1cm
\large
Damanvir Singh Binner \\
Department of Mathematics\\
Simon Fraser University \\
Burnaby, BC V5A 1S6\\
Canada \\
 dbinner@sfu.ca
\end{center}

\vskip .2in

\begin{abstract}
In a recent work, the present author generalized a fundamental result of Gauss related to quadratic reciprocity, and also showed that the above result of Gauss is equivalent to a special case of a well-known result of Sylvester related to the Frobenius coin problem. In this note, we use this equivalence to show that the above generalization of the result of Gauss naturally leads to an interesting generalization of the result of Sylvester. To be precise, for given positive coprime integers $a$ and $b$, and for a family of values of $k$ in the interval $0 \leq k < (a-1)(b-1)$, we find the number of nonnegative integers $\leq k$ that can be expressed in the form $ax+by$ for nonnegative integers $x$ and $y$. We also give an elementary proof of Eisenstein's Lemma for Jacobi symbols using floor function sums. Our proof provides a natural straightforward generalization of the Gauss-Eisenstein proof of the law of quadratic reciprocity for Jacobi symbols.
 \end{abstract}
 
 \section{Introduction}
 \label{Intro}

Throughout this note, $\lfloor x \rfloor$ denotes the greatest integer less than or equal to $x$. Recall the following well-known results of Gauss and Sylvester.

\begin{theorem}[Gauss (1808)]
\label{Quad}
For distinct odd primes $p$ and $q$, $$\sum_{i=1}^{\frac{p-1}{2}}\Big \lfloor\frac{iq}{p} \Big \rfloor+ \sum_{i=1}^{\frac{q-1}{2}}\Big \lfloor\frac{ip}{q}\Big \rfloor =  \frac{(p-1)(q-1)}{4}. $$
\end{theorem}

\begin{theorem}[Sylvester (1882)]
\label{Sylvester}
 If $a$ and $b$ are coprime numbers, the number of natural numbers that cannot be expressed in the form $ax + by$  for nonnegative integers $x$ and $y$ is equal to $\frac{(a-1)(b-1)}{2}$.
\end{theorem}

\begin{remark}
\label{odd}
Theorem \ref{Quad} and its proof hold verbatim for any odd positive coprime integers $a$ and $b$.
\end{remark}

In this note, the following special case of Theorem \ref{Sylvester} will be of particular interest.

\begin{theorem}
\label{Sylvester'}
 If $p$ and $q$ are distinct odd prime numbers, the number of natural numbers 
 that cannot be expressed in the form $px + qy$  for nonnegative integers $x$ and $y$ is equal to $\frac{(p-1)(q-1)}{2}$. 
\end{theorem}

We refer to Theorem \ref{Sylvester} as \emph{Sylvester's Theorem} and to Theorem \ref{Sylvester'} as \emph{Special Case of Sylvester's Theorem}.

Gauss \cite{Gauss} proved Theorem \ref{Quad} in $1808$, and this completed his third proof of the law of quadratic reciprocity. Eisenstein \cite{Eisenstein} gave a geometric proof of Theorem \ref{Quad} in $1844$. We refer the reader to Baumgart \cite[pp.\ $15$--$20$]{Classical} for more information about these classical proofs. Sylvester \cite{Sylvester82} proved Theorem \ref{Sylvester} in $1882$. In $1883$, he posed it as a recreational problem and Curran \cite{Sylvester} published a short proof based on generating functions. Several more contemporary proofs of Sylvester's Theorem are known \cite[Section 5.1]{alfonsin}.

 The present author \cite[Lemma 7]{Binner} generalized Theorem \ref{Quad} as follows.

  \begin{theorem}
  \label{GenReci}
  Let $a$, $b$, $d$, and $K$ be positive integers such that $b < a$, $d < a$, $\gcd(a,b)  = 1$, and $K = \Big \lfloor\frac{bd}{a} \Big\rfloor$. Then,  $$\sum_{i=1}^{d} \Big \lfloor\frac{ib}{a} \Big \rfloor+ \sum_{i=1}^{K} \Big \lfloor\frac{ia}{b} \Big \rfloor= dK.$$ 
  \end{theorem}
  
The present author \cite[Section 3]{Binner} proved that the Special Case of Sylvester's Theorem is equivalent to the reciprocity relation of Gauss in Theorem \ref{Quad}. As described below, the importance of this equivalence is demonstrated by the fact that we are able to generalize Sylvester's Theorem using our generalization of the Gauss' result (Theorem \ref{GenReci}). Prior to that, we choose an appropriate special case of Theorem \ref{GenReci} that is equivalent to Sylvester's Theorem. It turns out that the following special case of Theorem \ref{GenReci} is precisely what we require.

\begin{theorem}
\label{Strong}
Let $a$ and $b$ be positive coprime integers. Then,  $$\sum_{i=1}^{\left \lfloor\frac{a}{2} \right\rfloor} \Big \lfloor\frac{ib}{a} \Big \rfloor+ \sum_{i=1}^{\left \lfloor\frac{b}{2} \right\rfloor} \Big \lfloor\frac{ia}{b} \Big \rfloor= \Big \lfloor \frac{a}{2} \Big \rfloor \Big \lfloor \frac{b}{2} \Big \rfloor.$$ 
\end{theorem}

In Section \ref{Equiv}, we prove Theorem \ref{Strong} and also show that it is equivalent to Theorem \ref{Sylvester}. Note that Theorem \ref{GenReci} is a generalization of Theorem \ref{Strong}. Thus, it is natural to wonder whether the equivalence leads to a generalization of Theorem \ref{Sylvester}, that is equivalent to Theorem \ref{GenReci}. This is in fact true and leads to an interesting result (Theorem \ref{Best}). We make this sequence of equivalence of results more clear in Figure \ref{fig:Equival} below.

\begin{figure}[H]
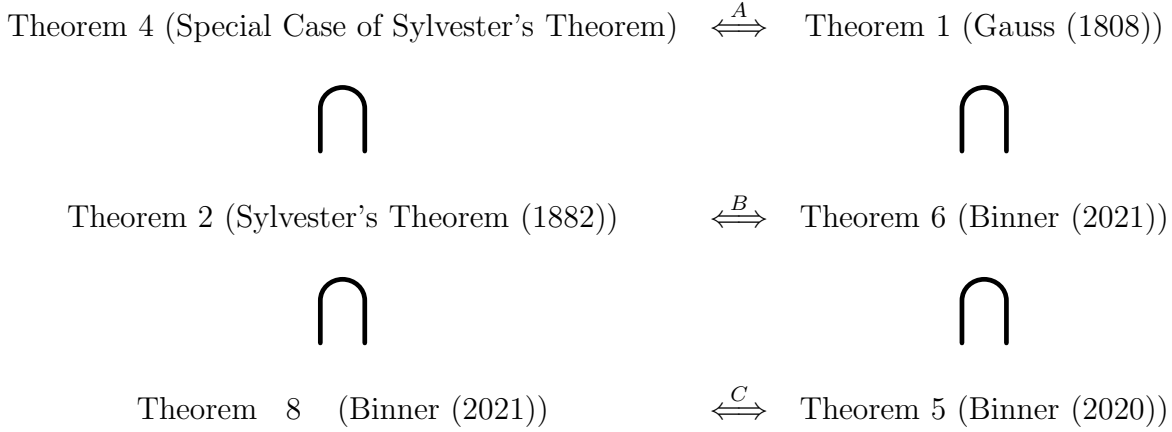

\begin{center}
	\begin{tabular}{ccc}
	\text{Theorem \ref{Sylvester'} (Special Case of Sylvester's Theorem}) & $\stackrel{A}{\Longleftrightarrow}$ & \text{Theorem} \ref{Quad} (\text{Gauss (1808)})\\
	\\
	$\scaleto{\bigcap}{5ex}$ & & $\scaleto{\bigcap}{5ex}$\\
	\\
	\text{Theorem \ref{Sylvester} (Sylvester's Theorem (1882)}) & $\stackrel{B}{\Longleftrightarrow}$  & \text{Theorem} \ref{Strong} (\text{Binner (2021)})\\
	\\
        $\scaleto{\bigcap}{5ex}$ & & $\scaleto{\bigcap}{5ex}$\\
        \\
        \text{Theorem} \hspace{.1cm} \ref{Best} \hspace{.2cm} (\text{Binner (2021)}) & $\stackrel{C}{\Longleftrightarrow}$ & \text{Theorem}  \ref{GenReci} (\text{Binner (2020)})
\end{tabular}
\end{center}
\caption{By Theorem X $\subset$ Theorem Y, we mean that Theorem X is a special case of Theorem Y, and by Theorem X $\Leftrightarrow$ Theorem Y, we mean that Theorems X and Y are equivalent. The present author \cite[Section 2.3]{Binner} proved Equivalence A in $2020$. We prove Equivalences B and C in Sections \ref{Equiv} and \ref{Sec4}, respectively.}
\label{fig:Equival}
\end{figure}
  
 Sylvester's Theorem shows that exactly half of nonnegative integers till the Frobenius number $ab-a-b$ can be expressed in the form $ax+by$. It is natural to ask a more general question.

\begin{question}
\label{NewSyl}
For given positive coprime integers $a$ and $b$, and given $k$ such that $0 \leq k < (a-1)(b-1)$, find the number of nonnegative integers $\leq k$ that can be expressed in the form $ax+by$ for nonnegative integers $x$ and $y$.
\end{question}

We denote this number by $N_0(a,b;k)$. For $k < 0$, we define $N_0(a,b;k)=0$. By Sylvester's Theorem, $N_0(a,b;ab-a-b) = \frac{(a-1)(b-1)}{2}$. In this note, we answer Question \ref{NewSyl} for a specific family of values of $k$ by proving Theorem \ref{Best}.

\begin{theorem}
\label{Best}
Let $a$ and $b$ be positive coprime integers with $b < a$. Further, let $0 < \alpha < a$ be such that $\alpha \equiv a$ (mod $2$), and $\beta = 2 \left \lfloor \frac{b(\alpha +a)}{2a} \right \rfloor - b $. Then, $$ N_0 \left(a,b;\frac{b\alpha + a\beta}{2}\right) = \frac{(\alpha +1)(\beta+1)}{2}. $$ 
\end{theorem}

Setting $\alpha =a-2$ in the above theorem gives Theorem \ref{Sylvester}. 
Note that $$ \beta \geq 2\left \lfloor \frac{b}{2}  \right \rfloor - b \geq -1. $$ Thus, for $\frac{b\alpha + a\beta}{2} < 0$, it must be true that $\beta = -1$, and then by Theorem \ref{Best}, $ N_0 \left(a,b;\frac{b\alpha + a\beta}{2}\right) = 0$, which is consistent with our definition of $N_0(a,b;k)$ for negative values of $k$.

We demonstrate Theorem \ref{Best} for an example. Suppose that $a=29$ and $b=23$. Then Theorem \ref{Sylvester} shows that $N_0(29,23;615) = 308$. However, Theorem \ref{Best} gives us $N_0(29,23;k)$ for a lot of values of $k$, as described in Table \ref{Table1}.

\begin{table}[H]
    \centering
    \begin{tabular}{|c|c|c|}
 \hline
$\alpha$ & $k$      &   $N_0(29,23;k)$ \\
\hline
  1 &  -1  & 0 \\
\hline
  3 & 49   & 4  \\
\hline
  5 & 101 &  12 \\
\hline
  7 & 153  & 24   \\
\hline
 9 & 205 &  40  \\
\hline
 11 & 228 & 48 \\
\hline
13 & 280 & 70 \\
\hline
15 & 332 & 96 \\
\hline
17 & 384 & 126 \\
\hline
19 & 436 & 160 \\
\hline
21 & 459 & 176 \\
\hline 
23 & 511 & 216 \\
\hline
25 & 563 & 260 \\
\hline
27 & 615 & 308 \\
\hline
    \end{tabular}
\caption{The values of $N_0(29,23;k)$ versus $k$, as $\alpha$ varies from $1$ to $27$ such that $\alpha$ is odd.}
\label{Table1}
\end{table}

In Section \ref{Sec4}, we prove Theorem \ref{Best} and also show that it is equivalent to Theorem \ref{GenReci} (Equivalence C in Figure \ref{fig:Equival}). For other values of $k$ not covered by Theorem \ref{Best}, we describe an easy method to calculate $N_0(a,b;k)$ in Section \ref{Otherk}.

We conclude our discussion on Question \ref{NewSyl} by mentioning that for given coprime natural numbers $a$ and $b$, the study of properties of numbers that cannot be expressed in the form $ax+by$ (also called \emph{nonrepresentable numbers}), where $x$ and $y$ are nonnegative integers, continues to be an active area of research. Let $NR(a,b)$ denote the set of nonnegative integers $n$ that cannot be expressed in the form $ax + by$. Then, by Sylvester's Theorem, $|NR(a,b)| = \frac{(a-1)(b-1)}{2}$. Brown and Shiue \cite{BS} discovered the sum $S(a,b)$ of nonrepresentable numbers. They proved that $$ S(a,b) := \sum_{n \in NR(a,b)} n = \frac{1}{12}(a-1)(b-1)(2ab-a-b-1). $$ R{\o}dseth considered a generalization of this sum $$ S_m(a,b) = \sum_{n \in NR(a,b)} n^m. $$ These sums $S_m(a,b)$ are commonly known as the \emph{Sylvester sums}. R{\o}dseth \cite{RS} discovered a formula for these sums in terms of Bernoulli numbers and found that $$ S_2(a,b) = \sum_{n \in NR(a,b)} n^2 = \frac{1}{12}(a-1)(b-1)ab(ab-a-b). $$ Recently, Komatsu and Zhang \cite{komatsu2} considered the \emph{weighted Sylvester sums} $$ S_m^{(\lambda)} = \sum_{n \in NR(a,b)} \lambda^{n-1} n^m. $$ They gave explicit expressions for these sums in terms of the Apostol-Bernoulli numbers.

However, in Question \ref{NewSyl}, we consider the number of nonrepresentable numbers below a given number $k$, instead of considering all the nonrepresentable numbers.

In Section \ref{Sec5}, we describe an application of floor function summation results to Jacobi symbols. Our approach leads to a natural straightforward generalization of the Gauss-Eisenstein proof of the law of quadratic reciprocity for Jacobi symbols.  We recall some main steps in the Gauss-Eisenstein proof of the law of quadratic reciprocity for Legendre symbols. More details about these results can be found in \cite[Chapter 3]{Niven}.

\begin{defn}
Let $a$ and $m$ be integers such that $\gcd(a,m)=1$. Then $a$ is called a \emph{quadratic residue modulo m} if the congruence $x^2 \equiv a$ (mod $m$) has a solution. If the congruence has no solution, then $a$ is called a \emph{quadratic nonresidue modulo m}.
\end{defn}

\begin{defn}
Let $p$ denote an odd prime. The \emph{Legendre symbol} $\left(\frac{a}{p} \right)$ is defined to be $1$ if $a$ is a quadratic residue modulo $p$, $-1$ if $a$ is a quadratic nonresidue modulo $p$, and $0$ if $p$ divides $a$.
\end{defn}

 \begin{theorem}[Gauss' Lemma for Legendre symbols]
 \label{GL}
 Suppose $p$ is an odd prime and $\gcd(a,p)=1$. Consider the integers $a, 2a, 3a, \cdots, \left( \frac{p-1}{2} \right) a$ and their least positive residues modulo $p$. If $n$ denotes the number of these residues that exceed $\frac{p}{2}$, then the Legendre symbol $\left(\frac{a}{p} \right)$ satisfies $ \left( \frac{a}{p} \right) = (-1)^n$. 
\end{theorem}

 \begin{theorem}[Eisenstein's Lemma for Legendre symbols]
\label{EL}
If $p$ is an odd prime and $a$ is any odd number not divisible by $p$, then $\left(\frac{a}{p} \right) = (-1)^t$ where $$t = \sum_{i=1}^{\frac{p-1}{2}} \left \lfloor \frac{ia}{p} \right \rfloor. $$
\end{theorem}

Theorems \ref{Quad} and \ref{EL} immediately lead to the law of quadratic reciprocity which is as follows:

\begin{theorem}[Law of quadratic reciprocity for Legendre symbols]
\label{QuadReci}
For distinct odd primes $p$ and $q$, $$ \left(\frac{p}{q}\right) \left(\frac{q}{p} \right) = (-1)^{\frac{(p-1)(q-1)}{4}}.$$
\end{theorem}

The theory of quadratic residues can be extended further by generalizing Legendre symbols to Jacobi symbols.  

\begin{defn}
Let $b$ be an odd number and $b = p_1^{r_1} p_2^{r_2} \cdots p_k^{r_k}$ be its prime factorization. Then the \emph{Jacobi symbol} $\left(\frac{a}{b} \right)$ is defined as $$ \left(\frac{a}{b} \right) := \left(\frac{a}{p_1} \right)^{r_1} \left(\frac{a}{p_2} \right)^{r_2} \cdots \left(\frac{a}{p_k} \right)^{r_k}, $$ where $\left(\frac{a}{p_i} \right)$ is the Legendre symbol.
\end{defn}

Schering \cite{Schering} generalized Gauss’ Lemma to the Jacobi symbol. However a direct proof of the Gauss-Schering Lemma \cite{Schering,Cartier,KK} seems to be quite technical. Zolotarev \cite{Zolo} observed that Legendre and Jacobi symbols are connected to signatures of naturally associated permutations. Using this approach, there are some other proofs \cite{DH,BC} showing that Gauss' Lemma can be generalized to the Jacobi symbol.
These methods give direct proofs of the law of quadratic reciprocity for Jacobi symbols at the cost of introducing some auxiliary concepts of an abstract algebraic nature. In the present note, we provide an elementary proof using floor function sums, showing that Eisenstein's Lemma also holds for Jacobi symbols. 

\begin{theorem}[Eisenstein's Lemma for Jacobi symbols]
    \label{Eisenstein}
    For odd positive coprime integers $a$ and $b$, the Jacobi symbol $ \left( \frac{a}{b} \right)$ is given as
 \begin{equation*}
  \left( \frac{a}{b} \right) = (-1)^{\mathlarger\sum\limits_{i=1}^{\frac{b-1}{2}} \left \lfloor \frac{ia}{b}\right \rfloor}.
  \end{equation*}
  \end{theorem}
  
  To prove Theorem \ref{Eisenstein} using Eisenstein's Lemma for Legendre symbols (Theorem \ref{EL}), it suffices to prove the following result.
  
  \begin{lemma}
  \label{GE1}
  For odd positive integers $a$, $b$ and $c$ such that $b$ and $c$ are coprime with $a$, $$ \sum_{i=1}^{\frac{bc-1}{2}} \left \lfloor \frac{ia}{bc}  \right \rfloor \equiv \sum_{i=1}^{\frac{b-1}{2}} \left \lfloor \frac{ia}{b} \right \rfloor + \sum_{i=1}^{\frac{c-1}{2}} \left \lfloor \frac{ia}{c} \right \rfloor \pmod 2 . $$
  \end{lemma}
  
  Using Theorem \ref{Quad} and Remark \ref{odd} (or Theorem \ref{Strong} instead), we can easily express all of the sums above in terms of summations of fractions having denominator $a$ and summation index $\frac{a-1}{2}$. From there, Lemma \ref{GE1} reduces to proving the following result.
  
  \begin{lemma}
  \label{GE2}
  For odd positive integers $a$, $b$ and $c$ such that $b$ and $c$ are coprime with $a$, $$ \sum_{i=1}^{\frac{a-1}{2}} \left \lfloor \frac{ibc}{a}  \right \rfloor \equiv \sum_{i=1}^{\frac{a-1}{2}} \left \lfloor \frac{ib}{a} \right \rfloor + \sum_{i=1}^{\frac{a-1}{2}} \left \lfloor \frac{ic}{a} \right \rfloor \pmod 2 . $$
  \end{lemma}
  
  We prove Lemma \ref{GE2} in Section \ref{Sec5}. The law of quadratic reciprocity for Jacobi symbols (Theorem \ref{QuadReci2} below) is then immediately obtained using Eisenstein's Lemma for Jacobi symbols (Theorem \ref{Eisenstein}), and Theorem \ref{Quad} and Remark \ref{odd} (or Theorem \ref{Strong} instead). 
  
  \begin{theorem}[Law of quadratic reciprocity for Jacobi symbols]
\label{QuadReci2}
If $a$ and $b$ are positive odd coprime integers, then $$ \Big( \frac{a}{b} \Big) \Big( \frac{b}{a} \Big) = (-1)^{\frac{(a-1)(b-1)}{4}}. $$
\end{theorem}
   
  \begin{remark}
  Using the standard techniques in the proof of Theorem \ref{EL} \cite[Theorem 3.3]{Niven}, it is straightforward to deduce Gauss' Lemma for Jacobi symbols from Eisenstein's Lemma for Jacobi symbols.
 \end{remark}

\section{Equivalence between Theorems \ref{Sylvester} and \ref{Strong}}
\label{Equiv}
In this section, we prove Theorem \ref{Strong} and show that it is equivalent to Theorem \ref{Sylvester} (Equivalence B in Figure \ref{fig:Equival}). 

 \begin{proof}[Proof of Theorem \ref{Strong}]
Without loss of generality, suppose that $b < a$. By setting the index of summation equal to $\left \lfloor \frac{a}{2} \right \rfloor$ in Theorem \ref{GenReci}, we get  
\begin{equation}
\label{Helping}
\sum_{i=1}^{\left \lfloor\frac{a}{2} \right\rfloor} \Big \lfloor\frac{ib}{a} \Big \rfloor+   \sum_{i=1}^{K} \Big \lfloor\frac{ia}{b} \Big \rfloor= K \Big \lfloor\frac{a}{2} \Big\rfloor,
\end{equation}
 where $$K =  \left \lfloor\frac{\left \lfloor\frac{a}{2} \right \rfloor b}{a} \right\rfloor. $$
We split the calculation into three cases based on the parity of $a$ and $b$. 

Case 1: Suppose $a$ is even, then $K = \left \lfloor\frac{b}{2} \right \rfloor$, and we are done.

Case 2: Suppose $a$ and $b$ are both odd, then $$ K = \Big \lfloor\frac{(a-1)b}{2a} \Big \rfloor= \Big \lfloor\frac{b-1}{2} + \frac{a-b}{2a} \Big \rfloor= \frac{b-1}{2} = \Big \lfloor \frac{b}{2} \Big \rfloor. $$ 

Case 3: Suppose $a$ is odd and $b$ is even, then $$ K = \Big \lfloor\frac{(a-1)b}{2a} \Big \rfloor = \Big \lfloor\frac{b}{2} - \frac{b}{2a} \Big \rfloor = \frac{b}{2} - 1 = \Big \lfloor\frac{b}{2} \Big \rfloor- 1,$$ and thus \eqref{Helping} becomes 

\begin{equation}
\label{Helping'}
\sum_{i=1}^{\left \lfloor\frac{a}{2} \right \rfloor} \Big \lfloor\frac{ib}{a} \Big \rfloor+   \sum_{i=1}^{\left \lfloor\frac{b}{2} \right \rfloor- 1} \Big \lfloor\frac{ia}{b} \Big \rfloor= \left(\Big \lfloor\frac{b}{2} \Big \rfloor- 1 \right) \Big \lfloor\frac{a}{2} \Big \rfloor.
\end{equation}

From \eqref{Helping'}, the theorem easily follows in this case.
\end{proof}

We establish the equivalence between Theorems \ref{Sylvester} and \ref{Strong} (Equivalence B in Figure \ref{fig:Equival}). Since the proofs of Lemmas \ref{EqnSum'}, \ref{Threetotwo'}, and \ref{SolnOther'} below are easy generalizations of our proofs in \cite[Lemma 10, Lemma 12, and Lemma 13]{Binner} respectively, we skip the details here. For the remainder of this section, suppose $a$ and $b$ are positive coprime integers.

\begin{lemma}
\label{EqnSum'}
The number of nonnegative integer solutions $(x,y,z)$ of the equation $ax + by + z =  b \Big \lfloor \frac{a}{2} \Big \rfloor$ is given by $$\Big \lfloor \frac{a}{2} \Big \rfloor+ 1 +\sum_{i=1}^{\left \lfloor \frac{a}{2} \right \rfloor} \Big \lfloor \frac{ib}{a} \Big \rfloor.$$ 
\end{lemma}

\begin{lemma}
\label{Threetotwo'}
The number of nonnegative integer solutions $(x,y,z)$ of the equation $$ax + by + z =  a \Big \lfloor \frac{b}{2} \Big \rfloor + b \Big \lfloor \frac{a}{2} \Big \rfloor$$ is equal to  $$a \Big \lfloor \frac{b}{2} \Big \rfloor + b \Big \lfloor \frac{a}{2} \Big \rfloor+1 - N_0,$$
where $N_0$ is the number of natural numbers which cannot be expressed as $ax + by$ for any nonnegative integers $x$ and $y$.
\end{lemma} 

\begin{lemma}
\label{SolnOther'}
The number of nonnegative integer solutions $(x,y,z)$ of the equation $$ax+by+z =   a \Big \lfloor\frac{b}{2} \Big \rfloor+ b \Big \lfloor\frac{a}{2} \Big \rfloor$$ is equal to
$$2 \left(\sum_{i=1}^{\left \lfloor\frac{a}{2} \right \rfloor} \Big \lfloor\frac{ib}{a} \Big \rfloor + \sum_{i=1}^{\left \lfloor\frac{b}{2} \right \rfloor} \Big \lfloor\frac{ia}{b} \Big \rfloor\right) + \Big \lfloor\frac{a}{2} \Big \rfloor+ \Big \lfloor\frac{b}{2} \Big \rfloor+ 1.$$
\end{lemma}

We are now ready to show the equivalence between Theorems \ref{Sylvester} and \ref{Strong} (Equivalence B in Figure \ref{fig:Equival}). Upon comparing the number of nonnegative integer solutions of the equation $ax+by+z =   a \Big \lfloor\frac{b}{2} \Big \rfloor+ b \Big \lfloor\frac{a}{2} \Big \rfloor$ obtained in Lemmas \ref{Threetotwo'} and \ref{SolnOther'}, and then using Lemma \ref{EqnSum'}, we find 

\begin{equation}
\label{One}
N_0 + 2 \left(\sum_{i=1}^{\left \lfloor\frac{a}{2} \right \rfloor} \Big \lfloor\frac{ib}{a} \Big \rfloor+ \sum_{i=1}^{\left \lfloor\frac{b}{2} \right \rfloor} \Big \lfloor\frac{ia}{b} \Big \rfloor\right)  =  (a-1) \Big \lfloor\frac{b}{2} \Big \rfloor + (b-1) \Big \lfloor\frac{a}{2} \Big \rfloor. 
\end{equation}

By taking three cases based on the parity of $a$ and $b$, it can be easily verified that 

\begin{equation}
\label{Two}
(a-1) \Big \lfloor\frac{b}{2} \Big \rfloor + (b-1) \Big \lfloor\frac{a}{2} \Big \rfloor = \frac{(a-1)(b-1)}{2} + 2 \Big \lfloor\frac{a}{2} \Big \rfloor \Big \lfloor\frac{b}{2} \Big\rfloor.
\end{equation}

From \eqref{One} and \eqref{Two}, we get that $$ N_0 + 2 \left(\sum_{i=1}^{\left \lfloor\frac{a}{2} \right \rfloor} \Big \lfloor\frac{ib}{a} \Big \rfloor+ \sum_{i=1}^{\left \lfloor\frac{b}{2} \right \rfloor} \Big \lfloor\frac{ia}{b} \Big \rfloor\right)  =  \frac{(a-1)(b-1)}{2} + 2 \Big \lfloor\frac{a}{2} \Big \rfloor\Big \lfloor\frac{b}{2} \Big \rfloor. $$

The equivalence between Theorems \ref{Sylvester} and \ref{Strong} (Equivalence B in Figure \ref{fig:Equival}) now readily follows.

\section{Proof of Theorem \ref{Best}}
\label{Sec4}

In this section, we prove Theorem \ref{Best} and the Equivalence C in Figure \ref{fig:Equival}. Recall that for given positive coprime integers $a$ and $b$, and given $k$ such that $0 \leq k < (a-1)(b-1)$, the symbol $N_0(a,b;k)$ denotes the number of natural numbers $\leq k$ that can be expressed in the form $ax+by$ for nonnegative integers $x$ and $y$. 

For positive coprime integers $a$ and $b$ and any natural number $n$, let $N(a,b;n)$ denote the number of nonnegative integer solutions of $ax+by=n$. An exact formula for $N(a,b;n)$ \cite{AT} is known. Further, it is well-known that $N(a,b;n+ab) = N(a,b;n)+1$ (see \cite[Lemma 1]{AT}). Using this fact while generalizing the proof of Lemma \ref{Threetotwo'}, we easily get the following result.

\begin{lemma}
\label{GiN}
Let $a$, $b$, $d$, and $K$ be positive integers such that $b < a$, $\frac{a}{2} < d < a$, $\gcd(a,b)  = 1$, and $K = \Big \lfloor\frac{bd}{a} \Big\rfloor$. The number of nonnegative integer solutions of the equation $$ax + by + z =  bd+aK$$ is equal to  $$ bd+aK+1-\frac{(a-1)(b-1)}{2} + N_0(a,b;bd+aK-ab). $$
\end{lemma} 

 Generalizing the proof of Lemma \ref{SolnOther'} and then using Theorem \ref{GenReci}, we get the following result.

\begin{lemma}
\label{GeN}
 Let $a$, $b$, $d$, and $K$ be positive integers such that $b < a$, $d < a$, $\gcd(a,b)  = 1$, and $K = \Big \lfloor\frac{bd}{a} \Big\rfloor$. The number of nonnegative integer solutions $(x,y,z)$ of the equation $$ax+by+z = bd+aK$$ is equal to $$2\left(\sum_{i=1}^{d} \left\lfloor \frac{ib}{a} \right\rfloor + \sum_{i=1}^{K} \left\lfloor \frac{ia}{b} \right\rfloor \right)+d+K+1. $$
\end{lemma}

The following lemma immediately follows from Lemmas \ref{GiN} and \ref{GeN}. 

\begin{lemma}
\label{Best1}
Let $a$, $b$, $d$, and $K$ be positive integers such that $b < a$, $\frac{a}{2} < d < a$, $\gcd(a,b)  = 1$, and $K = \Big \lfloor\frac{bd}{a} \Big\rfloor$. Then $$ N_0(a,b;bd+aK-ab) = 2\left(\sum_{i=1}^{d} \left\lfloor \frac{ib}{a} \right\rfloor + \sum_{i=1}^{K} \left\lfloor \frac{ia}{b} \right\rfloor - dK\right) + \frac{(2d-a+1)(2K-b+1)}{2}. $$
\end{lemma}

Using Lemma \ref{Best1}, it is clear that Theorem \ref{GenReci} is equivalent to the following theorem.

\begin{theorem}
\label{Best2}
Let $a$, $b$, $d$, and $K$ be positive integers such that $b < a$, $\frac{a}{2} < d < a$, $\gcd(a,b)  = 1$, and $K = \Big \lfloor\frac{bd}{a} \Big\rfloor$. Then $$ N_0(a,b;bd+aK-ab) = \frac{(2d-a+1)(2K-b+1)}{2}. $$
\end{theorem}

Observe that Theorem \ref{Best} is just another version of Theorem \ref{Best2} obtained by setting $\alpha = 2d-a$ and $\beta = 2K-b$ in Theorem \ref{Best2}. This completes the proof of Theorem \ref{Best} and its equivalence with Theorem \ref{GenReci} (Equivalence C in Figure \ref{fig:Equival}).

\section{$N_0(a,b;k)$ for other values of $k$}
\label{Otherk}

In this section, we describe an easy method to calculate $N_0(a,b;k)$ for other values of $k$ not covered by Theorem \ref{Best}. It is well-known that for $l < ab$, the equation $ax+by=l$ has at most one solution \cite[Lemma 2 and Lemma 4]{AT}. Using this fact, it is easy to see that for $k < ab$, $N_0(a,b;k)$ is equal to the number of nonnegative integer solutions of $ax+by+z=k$, which can be easily calculated using the algorithm described in \cite[Section 2.3]{Binner}. For example, suppose we want to calculate $N_0(29,23;257)$. Then, by \cite[Theorem 5]{Binner}, we get that $$ N_0(29,23;257) = 15 + \sum_{i=1}^8 \left \lfloor \frac{23i}{29} \right \rfloor + \sum_{i=1}^{18} \left \lfloor \frac{4i}{23} \right \rfloor . $$ By repeated applications of Theorem \ref{GenReci} and the division algorithm, as described in \cite[Section 2.3]{Binner}, we easily get that $\sum_{i=1}^8 \left \lfloor \frac{23i}{29} \right \rfloor = 24$ and $\sum_{i=1}^{18} \left \lfloor \frac{4i}{23} \right \rfloor = 21$, and thus $N_0(29,23;257) = 60$.

 \section{Proof of Lemma \ref{GE2}}
 \label{Sec5}  
  \begin{proof}[Proof of Lemma \ref{GE2}]
   For brevity of notation, let $r(m)$ denote the remainder when $m$ is divided by $a$. Since $a$, $b$ and $c$ are odd, we get 
  \begin{align}
  \label{E1}
   \sum_{i=1}^{\frac{a-1}{2}} \left \lfloor \frac{ibc}{a}  \right \rfloor &\equiv \sum_{i=1}^{\frac{a-1}{2}} a \left \lfloor \frac{ibc}{a}  \right \rfloor \pmod 2 \notag \\
   &= \sum_{i=1}^{\frac{a-1}{2}} ibc- r(ibc) \notag \\
   &\equiv \sum_{i=1}^{\frac{a-1}{2}} i +  \sum_{i=1}^{\frac{a-1}{2}} r(ibc) \pmod 2.
   \end{align}
   Next, we study the latter sum. Let $c_i$ denote the remainder when $ic$ is divided by $a$. Note that 
    \begin{align*}
    c_i &= ic - a \left \lfloor \frac{ic}{a} \right \rfloor \notag \\
    &\equiv i - \left \lfloor \frac{ic}{a} \right \rfloor \pmod 2. 
    \end{align*}
 Therefore, 
 \begin{align}
 \label{E2}
  r(ibc) &= r(bc_i) \notag \\
   &= bc_i - a \left \lfloor \frac{bc_i}{a} \right \rfloor \notag \\
   &\equiv c_i -  \left \lfloor \frac{bc_i}{a} \right \rfloor  \pmod 2 \notag \\
   &\equiv i - \left \lfloor \frac{ic}{a} \right \rfloor - \left \lfloor \frac{bc_i}{a} \right \rfloor \pmod 2. 
  \end{align}
   Thus, from \eqref{E1} and \eqref{E2}, we get that 
   \begin{equation*}
   \label{E3}
   \sum_{i=1}^{\frac{a-1}{2}} \left \lfloor \frac{ibc}{a}  \right \rfloor  \equiv  \sum_{i=1}^{\frac{a-1}{2}}  \left \lfloor \frac{ic}{a} \right \rfloor + \sum_{i=1}^{\frac{a-1}{2}} \left \lfloor \frac{bc_i}{a} \right \rfloor \pmod 2.
   \end{equation*}
   Therefore, to complete the proof of the lemma, it suffices to show that 
   \begin{equation}
   \label{E4}
    \sum_{i=1}^{\frac{a-1}{2}} \left \lfloor \frac{bc_i}{a} \right \rfloor \equiv \sum_{i=1}^{\frac{a-1}{2}} \left \lfloor \frac{ib}{a} \right \rfloor  \pmod 2. 
    \end{equation}
   Let $d_i = \min(c_i,a-c_i)$. Note that $$ \left \lfloor \frac{b(a-c_i)}{a} \right \rfloor = b-1- \left \lfloor \frac{bc_i}{a} \right \rfloor \equiv \left \lfloor \frac{bc_i}{a} \right \rfloor \pmod 2.$$ Therefore, for each $i$, $$  \left \lfloor \frac{bd_i}{a} \right \rfloor \equiv  \left \lfloor \frac{bc_i}{a} \right \rfloor \pmod 2. $$ Moreover, note that as $i$ varies from $1$ to $\frac{a-1}{2}$, so does $d_i$. Therefore, 
   \begin{align*}
 \sum_{i=1}^{\frac{a-1}{2}} \left \lfloor \frac{bc_i}{a} \right \rfloor &\equiv \sum_{i=1}^{\frac{a-1}{2}} \left \lfloor \frac{bd_i}{a} \right \rfloor \pmod 2 \\
 & = \sum_{i=1}^{\frac{a-1}{2}} \left \lfloor \frac{ib}{a} \right \rfloor \pmod 2,
   \end{align*}
   completing the proof of \eqref{E4}, and thus of Lemma \ref{GE2}. 
     \end{proof}

  \section{Acknowledgement}
\label{Acn}

I want to thank the anonymous referee for some excellent insights, especially regarding Jacobian quadratic reciprocity. I also wish to express my gratitude to  A. Rattan at SFU for some very helpful suggestions on the presentation of this paper. Finally, I want to thank the Maths Department at SFU for providing me various awards and fellowships which help me conduct my research.

\noindent 
2010 \emph{Mathematics Subject Classification}:~Primary 11D45;
Secondary  11A15, 11A07, 11D04, 05A15.

\medskip

\noindent 
\emph{Keywords}:  Frobenius coin problem,
nonrespresentable number,
Sylvester's result,
equivalence,
reciprocity relation,
floor function summation,
Gauss' Lemma for Jacobi symbol,
Eisenstein's Lemma for Jacobi symbol,
quadratic reciprocity for Jacobi symbol,
Frobenius number.

\end{document}